\documentclass[11pt]{article}
\usepackage{amsfonts,amscd,amssymb, amsmath}
\newtheorem{definition}{Definition}[section]

\newtheorem{proposition}[definition]{Proposition}

\newtheorem{theorem}[definition]{Theorem}

\newcommand{\gsl}{\mbox{${\;}$$>\hspace{-1.7mm}\triangleleft$${\;}$}}

\def\rawo\lonra{\longrightarrow}

\def\ot{\otimes}

\newcommand{\selabel}[1]{\label{se:#1}}

\allowdisplaybreaks[4]

\newenvironment{proof}{{\it Proof.}}{\hfill $ \square $ \vskip 4mm}

\begin{document}
\title{Invariance under twisting for crossed products
\thanks{Research partially supported by the CNCSIS project 
 ''Hopf algebras, cyclic homology and monoidal categories'', 
contract nr. 560/2009, CNCSIS code $ID_{-}69$.}}
\author {Florin Panaite\\
Institute of Mathematics of the 
Romanian Academy\\ 
PO-Box 1-764, RO-014700 Bucharest, Romania\\
e-mail: Florin.Panaite@imar.ro
}
\date{}
\maketitle

\begin{abstract}
We prove a result of the type ''invariance under twisting'' for  
Brzezi\'{n}ski's crossed products, as a common generalization of the 
invariance under twisting for twisted tensor products of algebras and 
the invariance under twisting for quasi-Hopf smash products. It turns out 
that this result contains also as a particular case the equivalence of crossed 
products by a coalgebra (due to Brzezi\'{n}ski). 
\end{abstract}
%%%%%%%%%%%%%%%%%%%%%%%%%%%%%%
\section*{Introduction}
%%%%%%%%%%%%%%%%%%%%%%%%%%%%%%
${\;\;\;\;}$If $A$ and $B$ are (associative unital) algebras and 
$R:B\ot A\rightarrow A\ot B$ is a linear map satisfying certain axioms 
(such an $R$ is called a {\em twisting map}) then $A\ot B$ becomes an 
associative unital algebra with a multiplication defined in terms of 
$R$ and the multiplications of $A$ and $B$; this algebra structure 
on $A\ot B$ is denoted by $A\ot _RB$ and called the 
{\em twisted tensor product} of $A$ and $B$ afforded by $R$ 
(cf. \cite{Cap}, \cite{VanDaele}). 
This construction appeared in a number of contexts and has 
various applications, see \cite{jlpvo} for a detailed discussion and 
references. Moreover, there are many concrete examples of 
twisted tensor products, such as the Hopf smash product 
and other kinds of products arising in Hopf algebra theory. 

A very general result about twisted tensor products of algebras was obtained 
in \cite{jlpvo}. It was directly inspired by the invariance under twisting of the 
Hopf smash product (and thus it was called {\em invariance under twisting} 
for twisted tensor products of algebras), but it contains also as 
particular cases a number of independent and previously unrelated 
results from Hopf algebra theory. This result states that, if 
$A\ot _RB$ is a twisted tensor product and on the vector space $B$ we have 
one more algebra structure denoted by $B'$ and we have also two 
linear maps $\theta , \gamma :B\rightarrow A\ot B$ satisfying a set of 
conditions, then one can define a new map $R':B'\ot A\rightarrow A\ot B'$ 
by a certain formula, this map turns out to be a twisting map and we 
have an algebra isomorphism $A\ot _{R'}B'\simeq A\ot _RB$.  

On the other hand, there exist important examples of 
''products'' of ''algebras'' that are {\em not} twisted tensor products, 
a prominent example being the classical Hopf crossed product. 
A very general construction, generalizing both the Hopf 
crossed product and the twisted tensor product of algebras, was 
introduced by  Brzezi\'{n}ski in \cite{brz}. Given an algebra $A$, 
a vector space $V$ endowed with a distinguished element $1_V$ 
and two linear maps $\sigma :V\ot V\rightarrow A\ot V$ and 
$R:V\ot A\rightarrow A\ot V$ satisfying certain conditions, 
Brzezi\'{n}ski defined an (associative unital) algebra structure 
on $A\ot V$, which will be denoted in what follows by 
$A\ot _{R, \sigma }V$ and called a {\em crossed product}. 
A twisted tensor product of algebras $A\ot _RB$ is the 
crossed product $A\ot _{R, \sigma }B$, where $\sigma :B\ot B
\rightarrow A\ot B$ is given by $\sigma (b, b')=1_A\ot bb'$. 
Another example of a crossed product, not discussed so far 
in the literature on this subject but important for us in 
what follows, is a smash product $H\# B$ between a 
quasi-bialgebra $H$ and a right $H$-module algebra $B$  
(this is the right-handed version of the smash product introduced 
in \cite{bpv}; since in general $B$ is {\em not} associative, 
in general $H\# B$ is {\em not} a twisted tensor product of algebras).   

The aim of this paper is to prove a result of the type invariance under 
twisting for  Brzezi\'{n}ski's crossed products. This result arose as a 
common generalization of the 
invariance under twisting for twisted tensor products of algebras and 
the invariance under twisting of the quasi-Hopf smash product from 
\cite{bpvo}. Namely, if  $A\ot _{R, \sigma }V$ is a crossed product 
and $\theta , \gamma :V\rightarrow A\ot V$ are linear maps, 
we can define certain maps  $\sigma ':V\ot V\rightarrow A\ot V$ and 
$R':V\ot A\rightarrow A\ot V$ and if some conditions are 
satisfied then  $A\ot _{R', \sigma '}V$ is a crossed product, isomorphic 
to  $A\ot _{R, \sigma }V$. After proving this result we show that 
it contains indeed as particular cases not only the invariance under twisting 
for twisted tensor products of algebras and the invariance under 
twisting of the quasi-Hopf smash product, but also another 
unrelated result, namely the equivalence of crossed products by a coalgebra 
proved by  Brzezi\'{n}ski (which in turn generalizes the equivalence 
of Hopf crossed products).  
%%%%%%%%%%%%%%%%%%%%%%%%%%%%%%%
\section{Preliminaries}\selabel{1}
%%%%%%%%%%%%%%%%%%%%%%%%%%%%%%%
${\;\;\;\;}$
We work over a commutative field $k$. All algebras, linear spaces
etc. will be over $k$; unadorned $\ot $ means $\ot_k$. By ''algebra'' we 
always mean an associative unital algebra.

We recall from \cite{Cap}, \cite{VanDaele} that, given two algebras $A$, $B$ 
and a $k$-linear map $R:B\ot A\rightarrow A\ot B$, with notation 
$R(b\ot a)=a_R\ot b_R$, for $a\in A$, $b\in B$, satisfying the conditions 
$a_R\otimes 1_R=a\otimes 1$, $1_R\otimes b_R=1\otimes b$, 
$(aa')_R\otimes b_R=a_Ra'_r\otimes b_{R_r}$, 
$a_R\otimes (bb')_R=a_{R_r}\otimes b_rb'_R$, 
for all $a, a'\in A$ and $b, b'\in B$ (where $r$ is another 
copy of $R$), if we define on $A\ot B$ a new multiplication, by 
$(a\ot b)(a'\ot b')=aa'_R\ot b_Rb'$, then this multiplication is associative 
with unit $1\ot 1$. In this case, the map $R$ is called 
a {\bf twisting map} between $A$ and $B$ and the new algebra 
structure on $A\ot B$ is denoted by $A\ot _RB$ and called the 
{\bf twisted tensor product} of $A$ and $B$ afforded by $R$. 

We recall from \cite{brz} the construction of  
Brzezi\'{n}ski's crossed product:
\begin{proposition} (\cite{brz}) \label{defbrz}
Let $(A, \mu , 1_A)$ be an (associative unital) algebra and $V$ a 
vector space equipped with a distinguished element $1_V\in V$. Then 
the vector space $A\ot V$ is an associative algebra with unit $1_A\ot 1_V$ 
and whose multiplication has the property that $(a\ot 1_V)(b\ot v)=
ab\ot v$, for all $a, b\in A$ and $v\in V$, if and only if there exist 
linear maps $\sigma :V\ot V\rightarrow A\ot V$ and 
$R:V\ot A\rightarrow A\ot V$  satisfying the following conditions:
\begin{eqnarray}
&&R(1_V\ot a)=a\ot 1_V, \;\;\;R(v\ot 1_A)=1_A\ot v, \;\;\;\forall 
\;a\in A, \;v\in V, \label{brz1} \\
&&\sigma (1_V, v)=\sigma (v, 1_V)=1_A\ot v, \;\;\;\forall 
\;v\in V, \label{brz2} \\
&&R\circ (id_V\ot \mu )=(\mu \ot id_V)\circ (id_A\ot R)\circ (R\ot id_A), 
\label{brz3} \\
&&(\mu \ot id_V)\circ (id_A\ot \sigma )\circ (R\ot id_V)\circ 
(id_V\ot \sigma ) \nonumber \\
&&\;\;\;\;\;\;\;\;\;\;
=(\mu \ot id_V)\circ (id_A\ot \sigma )\circ (\sigma \ot id_V), \label{brz4} \\
&&(\mu \ot id_V)\circ (id_A\ot \sigma )\circ (R\ot id_V)\circ 
(id_V\ot R ) \nonumber \\
&&\;\;\;\;\;\;\;\;\;\;
=(\mu \ot id_V)\circ (id_A\ot R )\circ (\sigma \ot id_A). \label{brz5} 
\end{eqnarray}
If this is the case, the multiplication of $A\ot V$ is given explicitely by
\begin{eqnarray*} 
&&\mu _{A\ot V}=(\mu _2\ot id_V)\circ (id_A\ot id_A\ot \sigma )\circ 
(id_A\ot R\ot id_V),
\end{eqnarray*}
where $\mu _2=\mu \circ (id_A\ot \mu )=\mu \circ (\mu \ot id_A)$. 
We denote by $A\ot _{R, \sigma }V$ this algebra structure and 
call it the {\bf crossed product} afforded by the data $(A, V, R, \sigma )$.  
\end{proposition}

If  $A\ot _{R, \sigma }V$ is a crossed product, we introduce the 
following Sweedler-type notation:
\begin{eqnarray*}
&&R:V\ot A\rightarrow A\ot V, \;\;\;R(v\ot a)=a_R\ot v_R, \\
&&\sigma :V\ot V\rightarrow A\ot V, \;\;\;\sigma (v, v')=\sigma _1(v, v') 
\ot \sigma _2(v, v'), 
\end{eqnarray*} 
for all $v, v'\in V$ and $a\in A$. With this notation, the multiplication of 
 $A\ot _{R, \sigma }V$ reads
\begin{eqnarray*}
&&(a\ot v)(a'\ot v')=aa'_R\sigma _1(v_R, v')\ot \sigma _2(v_R, v'), \;\;\;
\forall \;a, a'\in A, \;v, v'\in V.
\end{eqnarray*}

A twisted tensor product is a particular case of a crossed product 
(cf. \cite{guccione}), namely, if $A\ot _RB$ is a twisted tensor product of 
algebras then $A\ot _RB=A\ot _{R, \sigma }B$, where 
$\sigma :B\ot B\rightarrow A\ot B$ 
is given by $\sigma (b, b')=1_A\ot bb'$, for all $b, b'\in B$. 
%%%%%%%%%%%%%%%%%%%%%%%%%%%%%%
\section{The main result and its consequences}
%%%%%%%%%%%%%%%%%%%%%%%%%%%%%
\setcounter{equation}{0}
%%%%%%%%%%%%%%%%%%%%%%%%%%%%
${\;\;\;\;}$
We can state now the Invariance under twisting theorem for 
crossed products:
\begin{theorem} \label{main}
Let $A\ot _{R, \sigma }V$ be a crossed product and assume 
we are given two linear maps $\theta , \gamma :V\rightarrow A\ot V$, 
with notation $\theta (v)=v_{<-1>}\ot v_{<0>}$ and 
$\gamma (v)=v_{\{-1\}}\ot v_{\{0\}}$, for all $v\in V$. Define the maps 
$R':V\ot A\rightarrow A\ot V$ and $\sigma ':V\ot V\rightarrow A\ot V$ by 
the formulae
\begin{eqnarray*}
&&R'=(\mu _2\ot id_V)\circ (id_A\ot id_A\ot \gamma )\circ (id_A\ot R)\circ 
(\theta \ot id_A),   \\
&&\sigma '=(\mu \ot id_V)\circ (id_A\ot \gamma )\circ (\mu _2\ot id_V)
\circ (id_A\ot id_A\ot \sigma )\circ (id_A\ot R\ot id_V)\circ 
(\theta \ot \theta  ). 
\end{eqnarray*}
Assume that the following conditions are satisfied:
\begin{eqnarray}
&&\theta (1_V)=1_A\ot 1_V, \;\;\;\gamma (1_V)=1_A\ot 1_V, \label{cros1}\\
&&v_{<-1>}v_{<0>_{\{-1\}}}\ot v_{<0>_{\{0\}}}=1_A\ot v, 
\;\;\;\forall \;v\in V, \label{cros2}\\
&&v_{\{-1\}}v_{\{0\}_{<-1>}}\ot v_{\{0\}_{<0>}}=1_A\ot v, 
\;\;\;\forall \;v\in V, \label{cros3}\\
&&(\mu \ot id_V)\circ (\mu \ot \sigma ')\circ (id_A\ot \gamma \ot id_V)
\circ (R\ot id_V)\circ (id_V\ot \gamma )\nonumber \\
&&\;\;\;\;\;\;\;\;\;\;=(\mu \ot id_V)\circ (id_A\ot \gamma )\circ \sigma . 
\label{cros4}
\end{eqnarray}
Then  $A\ot _{R', \sigma '}V$ is a crossed product and we have an algebra 
isomorphism   $A\ot _{R', \sigma '}V\simeq A\ot _{R, \sigma }V$, 
$a\ot v\mapsto av_{<-1>}\ot v_{<0>}$. 
\end{theorem}
\begin{proof}
Let us note first that, in Sweedler-type notation, the maps $R'$ and 
$\sigma '$ are given by 
\begin{eqnarray*}
&&R'(v\ot a)=v_{<-1>}a_Rv_{<0>_{R_{\{-1\}}}}\ot 
v_{<0>_{R_{\{0\}}}}, \\
&&\sigma '(v, w)=v_{<-1>}w_{<-1>_R}\sigma _1(v_{<0>_R}, 
w_{<0>})\sigma _2(v_{<0>_R}, w_{<0>})_{\{-1\}}\ot 
\sigma _2(v_{<0>_R}, w_{<0>})_{\{0\}}.
\end{eqnarray*} 
We need to prove that the maps $R'$ and $\sigma '$ satisfy the 
conditions (\ref{brz1})--(\ref{brz5}). The conditions 
(\ref{brz1}) and (\ref{brz2}) are very easy to prove and are 
left to the reader, so we concentrate on 
(\ref{brz3})--(\ref{brz5}). We denote as usual by $R=r={\cal R}=
\overline{R}$ some more copies of $R$.\\
\underline{Proof of  (\ref{brz3})}: \\[2mm]
${\;\;\;\;\;}$$(\mu \ot id_V)\circ (id_A\ot R')\circ (R'\ot id_A)(v\ot a\ot a')$
\begin{eqnarray*}
&=&(\mu \ot id_V)\circ (id_A\ot R')(v_{<-1>}a_Rv_{<0>_{R_{\{-1\}}}}\ot 
v_{<0>_{R_{\{0\}}}}\ot a')\\
&=&v_{<-1>}a_Rv_{<0>_{R_{\{-1\}}}} 
v_{<0>_{R_{\{0\}_{<-1>}}}}a'_rv_{<0>_{R_{\{0\}_{<0>_{r_{\{-1\}}}}}}}
\ot v_{<0>_{R_{\{0\}_{<0>_{r_{\{0\}}}}}}}\\
&\overset{(\ref{cros3})}{=}&v_{<-1>}a_Ra'_r
v_{<0>_{R_{r_{\{-1\}}}}}\ot v_{<0>_{R_{r_{\{0\}}}}}\\
&\overset{(\ref{brz3})}{=}&v_{<-1>}(aa')_R
v_{<0>_{R_{\{-1\}}}}\ot v_{<0>_{R_{\{0\}}}}\\
&=&R'\circ (id_V\ot \mu )(v\ot a\ot a'), \;\;\;q.e.d.
\end{eqnarray*}
\underline{Proof of  (\ref{brz4})}: \\[2mm]
Note first that (\ref{brz4}) and (\ref{brz5}) for $R$, $\sigma $ 
may be written 
in Sweedler-type notation as 
\begin{eqnarray}
&&\sigma _1(y, z)_R\sigma _1(x_R, \sigma _2(y, z))\ot 
\sigma _2(x_R, \sigma _2(y, z)) \nonumber \\
&&\;\;\;\;\;\;\;\;\;\;=\sigma _1(x, y)\sigma _1(\sigma _2(x, y), z)\ot 
\sigma _2(\sigma _2(x, y), z), \label{brz4'}\\
&&a_{R_r}\sigma _1(v_r, w_R)\ot \sigma _2(v_r, w_R)
=\sigma _1(v, w)a_R\ot \sigma _2(v, w)_R, \label{brz5'}
\end{eqnarray}
for all $a\in A$ and $x, y, z, v, w\in V$. Now we compute:\\[2mm]
${\;\;\;\;\;}$$(\mu \ot id_V)\circ (id_A\ot \sigma ' )\circ (R'\ot id_V)\circ 
(id_V\ot \sigma ')(x\ot y\ot z)$
\begin{eqnarray*}
&=&(\mu \ot id_V)\circ (id_A\ot \sigma ' )\circ (R'\ot id_V)
(x\ot y_{<-1>}z_{<-1>_R}\sigma _1(y_{<0>_R}, 
z_{<0>})\\
&&\sigma _2(y_{<0>_R}, z_{<0>})_{\{-1\}}\ot 
\sigma _2(y_{<0>_R}, z_{<0>})_{\{0\}})\\
&=&(\mu \ot id_V)\circ (id_A\ot \sigma ' )
(x_{<-1>}[y_{<-1>}z_{<-1>_R}\sigma _1(y_{<0>_R}, 
z_{<0>})\sigma _2(y_{<0>_R}, z_{<0>})_{\{-1\}}]_r\\
&&x_{<0>_{r_{\{-1\}}}}\ot x_{<0>_{r_{\{0\}}}}
\ot \sigma _2(y_{<0>_R}, z_{<0>})_{\{0\}})\\
&=&x_{<-1>}[y_{<-1>}z_{<-1>_R}\sigma _1(y_{<0>_R}, 
z_{<0>})\sigma _2(y_{<0>_R}, z_{<0>})_{\{-1\}}]_r
x_{<0>_{r_{\{-1\}}}}x_{<0>_{r_{\{0\}_{<-1>}}}}\\
&&\sigma _2(y_{<0>_R}, z_{<0>})_{\{0\}_{<-1>_{\cal R}}}
\sigma _1(x_{<0>_{r_{\{0\}_{<0>_{\cal R}}}}}, 
\sigma _2(y_{<0>_R}, z_{<0>})_{\{0\}_{<0>}})\\
&&\sigma _2(x_{<0>_{r_{\{0\}_{<0>_{\cal R}}}}}, 
\sigma _2(y_{<0>_R}, z_{<0>})_{\{0\}_{<0>}})_{\{-1\}}\\
&&\ot 
\sigma _2(x_{<0>_{r_{\{0\}_{<0>_{\cal R}}}}}, 
\sigma _2(y_{<0>_R}, z_{<0>})_{\{0\}_{<0>}})_{\{0\}}\\
&\overset{(\ref{cros3})}{=}&x_{<-1>}[y_{<-1>}z_{<-1>_R}
\sigma _1(y_{<0>_R}, 
z_{<0>})\sigma _2(y_{<0>_R}, z_{<0>})_{\{-1\}}]_r\\
&&\sigma _2(y_{<0>_R}, z_{<0>})_{\{0\}_{<-1>_{\cal R}}}
\sigma _1(x_{<0>_{r_{\cal R}}}, 
\sigma _2(y_{<0>_R}, z_{<0>})_{\{0\}_{<0>}})\\
&&\sigma _2(x_{<0>_{r_{\cal R}}}, 
\sigma _2(y_{<0>_R}, z_{<0>})_{\{0\}_{<0>}})_{\{-1\}}\ot 
\sigma _2(x_{<0>_{r_{\cal R}}}, 
\sigma _2(y_{<0>_R}, z_{<0>})_{\{0\}_{<0>}})_{\{0\}}\\
&\overset{(\ref{brz3})}{=}&x_{<-1>}[y_{<-1>}z_{<-1>_R}
\sigma _1(y_{<0>_R}, z_{<0>})]_{\overline{R}}
\sigma _2(y_{<0>_R}, z_{<0>})_{\{-1\}_r}\\
&&\sigma _2(y_{<0>_R}, z_{<0>})_{\{0\}_{<-1>_{\cal R}}}
\sigma _1(x_{<0>_{\overline{R}_{r_{\cal R}}}}, 
\sigma _2(y_{<0>_R}, z_{<0>})_{\{0\}_{<0>}})\\
&&\sigma _2(x_{<0>_{\overline{R}_{r_{\cal R}}}}, 
\sigma _2(y_{<0>_R}, z_{<0>})_{\{0\}_{<0>}})_{\{-1\}}\ot 
\sigma _2(x_{<0>_{\overline{R}_{r_{\cal R}}}}, 
\sigma _2(y_{<0>_R}, z_{<0>})_{\{0\}_{<0>}})_{\{0\}}\\
&\overset{(\ref{brz3})}{=}&x_{<-1>}[y_{<-1>}z_{<-1>_R}
\sigma _1(y_{<0>_R}, z_{<0>})]_{\overline{R}}\\
&&[\sigma _2(y_{<0>_R}, z_{<0>})_{\{-1\}}
\sigma _2(y_{<0>_R}, z_{<0>})_{\{0\}_{<-1>}}]_r\\
&&\sigma _1(x_{<0>_{\overline{R}_r}}, 
\sigma _2(y_{<0>_R}, z_{<0>})_{\{0\}_{<0>}})
\sigma _2(x_{<0>_{\overline{R}_r}}, 
\sigma _2(y_{<0>_R}, z_{<0>})_{\{0\}_{<0>}})_{\{-1\}}\\
&&\ot 
\sigma _2(x_{<0>_{\overline{R}_r}}, 
\sigma _2(y_{<0>_R}, z_{<0>})_{\{0\}_{<0>}})_{\{0\}}\\
&\overset{(\ref{cros3}), (\ref{brz1})}{=}&
x_{<-1>}[y_{<-1>}z_{<-1>_R}
\sigma _1(y_{<0>_R}, z_{<0>})]_{\overline{R}}
\sigma _1(x_{<0>_{\overline{R}}}, 
\sigma _2(y_{<0>_R}, z_{<0>}))\\
&&\sigma _2(x_{<0>_{\overline{R}}}, 
\sigma _2(y_{<0>_R}, z_{<0>}))_{\{-1\}}\ot 
\sigma _2(x_{<0>_{\overline{R}}}, 
\sigma _2(y_{<0>_R}, z_{<0>}))_{\{0\}}\\
&\overset{(\ref{brz3})}{=}&
x_{<-1>}y_{<-1>_{\overline{R}}}z_{<-1>_{R_r}}
\sigma _1(y_{<0>_R}, z_{<0>})_{\cal R}
\sigma _1(x_{<0>_{\overline{R}_{r_{\cal R}}}}, 
\sigma _2(y_{<0>_R}, z_{<0>}))\\
&&\sigma _2(x_{<0>_{\overline{R}_{r_{\cal R}}}}, 
\sigma _2(y_{<0>_R}, z_{<0>}))_{\{-1\}}\ot 
\sigma _2(x_{<0>_{\overline{R}_{r_{\cal R}}}}, 
\sigma _2(y_{<0>_R}, z_{<0>}))_{\{0\}}\\
&\overset{(\ref{brz4'})}{=}&
x_{<-1>}y_{<-1>_{\overline{R}}}z_{<-1>_{R_r}}
\sigma _1(x_{<0>_{\overline{R}_r}}, y_{<0>_R})
\sigma _1(\sigma _2(x_{<0>_{\overline{R}_r}}, y_{<0>_R}), 
z_{<0>})\\
&&\sigma _2(\sigma _2(x_{<0>_{\overline{R}_r}}, y_{<0>_R}), 
z_{<0>})_{\{-1\}}
\ot \sigma _2(\sigma _2(x_{<0>_{\overline{R}_r}}, y_{<0>_R}), 
z_{<0>})_{\{0\}},
\end{eqnarray*}
${\;\;\;\;\;}$$(\mu \ot id_V)\circ (id_A\ot \sigma ')\circ 
(\sigma '\ot id_V)(x\ot y\ot z)$
\begin{eqnarray*}
&=&(\mu \ot id_V)\circ (id_A\ot \sigma ')
(x_{<-1>}y_{<-1>_{\overline{R}}}\sigma _1(x_{<0>_{\overline{R}}}, 
y_{<0>})\sigma _2(x_{<0>_{\overline{R}}}, y_{<0>})_{\{-1\}}\\
&&\ot 
\sigma _2(x_{<0>_{\overline{R}}}, y_{<0>})_{\{0\}}\ot z)\\
&=&x_{<-1>}y_{<-1>_{\overline{R}}}\sigma _1(x_{<0>_{\overline{R}}}, 
y_{<0>})\sigma _2(x_{<0>_{\overline{R}}}, y_{<0>})_{\{-1\}}
\sigma _2(x_{<0>_{\overline{R}}}, y_{<0>})_{\{0\}_{<-1>}}\\
&&z_{<-1>_R}\sigma _1(\sigma _2(x_{<0>_{\overline{R}}}, y_{<0>})
_{\{0\}_{<0>_R}}, z_{<0>})
\sigma _2(\sigma _2(x_{<0>_{\overline{R}}}, y_{<0>})
_{\{0\}_{<0>_R}}, z_{<0>})_{\{-1\}}\\
&&\ot \sigma _2(\sigma _2(x_{<0>_{\overline{R}}}, y_{<0>})
_{\{0\}_{<0>_R}}, z_{<0>})_{\{0\}}\\
&\overset{(\ref{cros3})}{=}&
x_{<-1>}y_{<-1>_{\overline{R}}}\sigma _1(x_{<0>_{\overline{R}}}, 
y_{<0>})z_{<-1>_R}\sigma _1
(\sigma _2(x_{<0>_{\overline{R}}}, y_{<0>})
_R, z_{<0>})\\
&&\sigma _2(\sigma _2(x_{<0>_{\overline{R}}}, y_{<0>})
_R, z_{<0>})_{\{-1\}}
\ot \sigma _2(\sigma _2(x_{<0>_{\overline{R}}}, y_{<0>})
_R, z_{<0>})_{\{0\}}\\
&\overset{(\ref{brz5'})}{=}&
x_{<-1>}y_{<-1>_{\overline{R}}}z_{<-1>_{R_r}}
\sigma _1(x_{<0>_{\overline{R}_r}}, y_{<0>_R})
\sigma _1(\sigma _2(x_{<0>_{\overline{R}_r}}, y_{<0>_R}), 
z_{<0>})\\
&&\sigma _2(\sigma _2(x_{<0>_{\overline{R}_r}}, y_{<0>_R}), 
z_{<0>})_{\{-1\}}
\ot \sigma _2(\sigma _2(x_{<0>_{\overline{R}_r}}, y_{<0>_R}), 
z_{<0>})_{\{0\}},
\end{eqnarray*}
and we see that the two terms coincide.\\
\underline{Proof of  (\ref{brz5})}: \\[2mm]
Let $v, w\in V$ and $a\in A$; we compute:\\[2mm]
${\;\;\;\;\;}$$(\mu \ot id_V)\circ (id_A\ot \sigma ')\circ (R'\ot id_V)\circ 
(id_V\ot R') (v\ot w\ot a)$
\begin{eqnarray*}
&=&(\mu \ot id_V)\circ (id_A\ot \sigma ')\circ (R'\ot id_V)
(v\ot w_{<-1>}a_rw_{<0>_{r_{\{-1\}}}}\ot 
w_{<0>_{r_{\{0\}}}})\\
&=&(\mu \ot id_V)\circ (id_A\ot \sigma ')(v_{<-1>}
(w_{<-1>}a_rw_{<0>_{r_{\{-1\}}}})_Rv_{<0>_{R_{\{-1\}}}}
\ot v_{<0>_{R_{\{0\}}}}\\
&&\ot  w_{<0>_{r_{\{0\}}}})\\
&=&v_{<-1>}
(w_{<-1>}a_rw_{<0>_{r_{\{-1\}}}})_Rv_{<0>_{R_{\{-1\}}}}
\sigma '_1(v_{<0>_{R_{\{0\}}}}, w_{<0>_{r_{\{0\}}}})\\
&&\ot \sigma '_2(v_{<0>_{R_{\{0\}}}}, w_{<0>_{r_{\{0\}}}})\\
&\overset{(\ref{brz3})}{=}&
v_{<-1>}
w_{<-1>_R}a_{r_{\cal R}}w_{<0>_{r_{\{-1\}_{\overline{R}}}}}
v_{<0>_{R_{{\cal R}_{\overline{R}_{\{-1\}}}}}}
\sigma '_1(v_{<0>_{R_{{\cal R}_{\overline{R}_{\{0\}}}}}}, 
w_{<0>_{r_{\{0\}}}})\\
&&\ot \sigma '_2(v_{<0>_{R_{{\cal R}_{\overline{R}_{\{0\}}}}}}, 
w_{<0>_{r_{\{0\}}}})\\
&\overset{(\ref{cros4})}{=}&
v_{<-1>}w_{<-1>_R}a_{r_{\cal R}}
\sigma _1(v_{<0>_{R_{\cal R}}}, w_{<0>_r})
\sigma _2(v_{<0>_{R_{\cal R}}}, w_{<0>_r})_{\{-1\}}\\
&&\ot \sigma _2(v_{<0>_{R_{\cal R}}}, w_{<0>_r})_{\{0\}}, 
\end{eqnarray*}
${\;\;\;\;\;}$$(\mu \ot id_V)\circ (id_A\ot R')\circ (\sigma '\ot id_A)
(v\ot w\ot a)$
\begin{eqnarray*}
&=&(\mu \ot id_V)\circ (id_A\ot R')(v_{<-1>}w_{<-1>_R}
\sigma _1(v_{<0>_R}, 
w_{<0>})\sigma _2(v_{<0>_R}, w_{<0>})_{\{-1\}}\\
&&\ot 
\sigma _2(v_{<0>_R}, w_{<0>})_{\{0\}}\ot a)\\
&=&v_{<-1>}w_{<-1>_R}
\sigma _1(v_{<0>_R}, 
w_{<0>})\sigma _2(v_{<0>_R}, w_{<0>})_{\{-1\}}
\sigma _2(v_{<0>_R}, w_{<0>})_{\{0\}_{<-1>}}\\
&&a_r\sigma _2(v_{<0>_R}, w_{<0>})_{\{0\}_{<0>_{r_{\{-1\}}}}}\ot 
\sigma _2(v_{<0>_R}, w_{<0>})_{\{0\}_{<0>_{r_{\{0\}}}}}\\
&\overset{(\ref{cros3})}{=}&
v_{<-1>}w_{<-1>_R}\sigma _1(v_{<0>_R}, w_{<0>})
a_r\sigma _2(v_{<0>_R}, w_{<0>})_{r_{\{-1\}}}\ot 
\sigma _2(v_{<0>_R}, w_{<0>})_{r_{\{0\}}}\\
&\overset{(\ref{brz5'})}{=}&
v_{<-1>}w_{<-1>_R}a_{r_{\cal R}}
\sigma _1(v_{<0>_{R_{\cal R}}}, w_{<0>_r})
\sigma _2(v_{<0>_{R_{\cal R}}}, w_{<0>_r})_{\{-1\}}\\
&&\ot \sigma _2(v_{<0>_{R_{\cal R}}}, w_{<0>_r})_{\{0\}},
\end{eqnarray*}
and we see that the two terms coincide. Thus, 
$A\ot _{R', \sigma '}V$ is indeed a crossed product. 

We prove now that the map $\varphi :A\ot _{R', \sigma '}V
\rightarrow A\ot _{R, \sigma }V$, $\varphi (a\ot v)=
av_{<-1>}\ot v_{<0>}$, is an algebra isomorphism. 
First, using (\ref{cros2}) and (\ref{cros3}), it is easy to see that 
$\varphi $ is bijective, with inverse given by 
$a\ot v\mapsto av_{\{-1\}}\ot v_{\{0\}}$. It is 
obvious that $\varphi (1\ot 1)=1\ot 1$, so we 
only have to prove that $\varphi$ is multiplicative. We compute: 
\begin{eqnarray*}
\varphi ((a\ot v)(a'\ot v'))&=&
\varphi (aa'_{R'}\sigma '_1(v_{R'}, v')\ot \sigma '_2(v_{R'}, v')\\
&=&\varphi (av_{<-1>}a'_Rv_{<0>_{R_{\{-1\}}}}
\sigma '_1(v_{<0>_{R_{\{0\}}}}, v')\ot 
\sigma '_2(v_{<0>_{R_{\{0\}}}}, v'))\\
&=&\varphi (av_{<-1>}a'_Rv_{<0>_{R_{\{-1\}}}}
v_{<0>_{R_{\{0\}_{<-1>}}}}v'_{<-1>_r}
\sigma _1(v_{<0>_{R_{\{0\}_{<0>_r}}}}, v'_{<0>})\\
&&\sigma _2(v_{<0>_{R_{\{0\}_{<0>_r}}}}, v'_{<0>})_{\{-1\}}
\ot \sigma _2(v_{<0>_{R_{\{0\}_{<0>_r}}}}, v'_{<0>})_{\{0\}})\\
&\overset{(\ref{cros3})}{=}&
\varphi (av_{<-1>}a'_Rv'_{<-1>_r}
\sigma _1(v_{<0>_{R_r}}, v'_{<0>})\\
&&\sigma _2(v_{<0>_{R_r}}, v'_{<0>})_{\{-1\}}
\ot \sigma _2(v_{<0>_{R_r}}, v'_{<0>})_{\{0\}})\\
&=&av_{<-1>}a'_Rv'_{<-1>_r}
\sigma _1(v_{<0>_{R_r}}, v'_{<0>})
\sigma _2(v_{<0>_{R_r}}, v'_{<0>})_{\{-1\}}\\
&&\sigma _2(v_{<0>_{R_r}}, v'_{<0>})_{\{0\}_{<-1>}}
\ot \sigma _2(v_{<0>_{R_r}}, v'_{<0>})_{\{0\}_{<0>}}\\
&\overset{(\ref{cros3})}{=}&
av_{<-1>}a'_Rv'_{<-1>_r}
\sigma _1(v_{<0>_{R_r}}, v'_{<0>})\ot 
\sigma _2(v_{<0>_{R_r}}, v'_{<0>})\\
&\overset{(\ref{brz3})}{=}&
av_{<-1>}(a'v'_{<-1>})_R
\sigma _1(v_{<0>_R}, v'_{<0>})\ot 
\sigma _2(v_{<0>_R}, v'_{<0>})\\
&=&(av_{<-1>}\ot v_{<0>})(a'v'_{<-1>}\ot v'_{<0>})\\
&=&\varphi (a\ot v)\varphi (a'\ot v'), 
\end{eqnarray*}
finishing the proof.
\end{proof}

We explain now that Theorem \ref{main} generalizes indeed the 
Invariance under twisting for twisted tensor products 
of algebras proved in Theorem 4.8 in \cite{jlpvo}, which we first recall: 
\begin{theorem} (\cite{jlpvo}) \label{invttp}
Let $A\ot _RB$ be a twisted tensor product of algebras and denote
the multiplication of $B$ by $b\ot b'\mapsto bb'$. Assume that on
the vector space $B$ we have one more algebra structure, denoted by
$B'$, with the same unit as $B$ and multiplication denoted by $b\ot
b'\mapsto b*b'$. Assume that we are given two linear maps
$\theta ,\gamma :B\rightarrow A\ot B$, with notation $\theta
(b)=b_{<-1>}\ot b_{<0>}$ and $\gamma (b)=b_{\{-1\}}\ot b_{\{0\}}$,
such that $\theta $ is an algebra map from $B'$ to $A\ot _RB$,
$\gamma (1)=1\ot 1$ and the following relations 
are satisfied, for all $b, b'\in B$:
\begin{eqnarray}
&&\gamma (bb')=b'_{\{-1\}_R}b_{R_{\{-1\}}}\ot
b_{R_{\{0\}}}*b'_{\{0\}},
\label{rel1} \\
&&b_{<-1>}b_{<0>_{\{-1\}}}\ot b_{<0>_{\{0\}}}=1\ot b, \label{rel2}\\
&&b_{\{-1\}}b_{\{0\}_{<-1>}}\ot b_{\{0\}_{<0>}}=1\ot b. \label{rel3}
\end{eqnarray}
Then the map $R':B'\ot A\rightarrow A\ot B'$, $R'(b\ot a)=b_{<-1>}a_R
b_{<0>_{R_{\{-1\}}}}\ot b_{<0>_{R_{\{0\}}}}$, 
is a twisting map and we have an algebra isomorphism
$A\ot _{R'}B'\simeq A\ot _RB, \;\;a\ot b\mapsto ab_{<-1>}\ot
b_{<0>}$.
\end{theorem}

We want to see how Theorem \ref{main} generalizes Theorem \ref{invttp}. 
We begin with the data $A$, $B$, $B'$, $R$, $\theta $, $\gamma $ as in the 
hypothesis of Theorem \ref{invttp}, consider the map  
$\sigma :B\ot B\rightarrow A\ot B$, $\sigma (b, b')=1_A\ot bb'$, 
for all $b, b'\in B$,  
so we have the crossed product  $A\ot _{R, \sigma }B$ 
with $A\ot _RB=A\ot _{R, \sigma }B$ and the maps $\theta , \gamma $ and 
so we can define the maps $R'$ and $\sigma '$ as in Theorem \ref{main} 
(of course, $V$ in Theorem \ref{main} is B as vector spaces). 
Obviously, the formula for $R'$ in Theorem \ref{main} is the same 
as the one for $R'$ in Theorem \ref{invttp}. We want to see how 
$\sigma '$ looks like. By using the fact that $\theta $ is an algebra 
map from $B'$ to $A\ot _RB$ and the formula 
$\sigma (b, b')=1_A\ot bb'$, an easy computation shows that we have 
$\sigma '(b, b')=1_A\ot b*b'$, for all $b, b'\in B$. Then one can 
easily see that the relations (\ref{cros1})--(\ref{cros4}) 
hold (note that (\ref{cros4}) reduces to (\ref{rel1})). Thus, 
all hypotheses of Theorem \ref{main} are fulfilled, so we 
have the crossed product $A\ot _{R', \sigma '}B$ and the algebra 
isomorphism $A\ot _{R', \sigma '}B\simeq A\ot _{R, \sigma }B$. But since 
$\sigma '$ is defined by $\sigma '(b, b')=1_A\ot b*b'$, it is obvious that 
the multiplication of $A\ot _{R', \sigma '}B$ is exactly the one of 
$A\ot _{R'}B'$, i.e. $A\ot _{R', \sigma '}B\equiv A\ot _{R'}B'$ and so 
we obtain the algebra isomorphism $A\ot _{R'}B'\simeq A\ot _RB$ 
as in Theorem \ref{invttp}.

We explain now that Theorem \ref{main} generalizes the invariance 
under twisting for (right) quasi-Hopf smash products (\cite{bpvo}), 
which we recall first (we use terminology and notation as 
in \cite{bpvo}).

Let $H$ be a quasi-bialgebra with associator $\Phi $ 
and $F\in H\ot H$ a gauge transformation, 
with notation $F=F^1\ot F^2$ and $F^{-1}=G^1\ot G^2$. 
Consider the Drinfeld twist of $H$, denoted by $H_F$, which is a 
quasi-bialgebra having the same underlying vector space, 
multiplication, unit and counit as $H$ and comultiplication and 
associator defined by 
\begin{eqnarray*}
&&\Delta _F(h)=F\Delta (h)F^{-1}, \;\;\;
\Phi_F=(1\ot F)(id \ot \Delta )(F) \Phi (\Delta \ot id)
(F^{-1})(F^{-1}\ot 1).
\end{eqnarray*}
We denote in what follows by $\Delta _F(h)=h_{(1)}\ot h_{(2)}$, 
$\Phi _F=\tilde{X}^1\ot \tilde{X}^2\ot \tilde{X}^3$, 
$\Phi _F^{-1}=\tilde{x}^1\ot \tilde{x}^2\ot \tilde{x}^3$. 

Let $B$ be a right $H$-module algebra (an algebra in the monoidal 
category of right $H$-modules), 
so we can consider the (right) smash product $H\#B$, which is an 
associative algebra having 
$H\ot B$ as underlying vector space, multiplication 
$(h\#b)(h'\#b')=hh'_1x^1\# (b\cdot h'_2x^2)(b'\cdot x^3)$ and unit 
$1_H\# 1_B$ (we denoted as usual $\Phi =X^1\ot X^2\ot X^3$ 
and $\Phi ^{-1}=x^1\ot x^2\ot x^3$ the associator of 
$H$ and its inverse).
\begin{proposition} (\cite{bpvo}) \label{invquasi}
If we introduce on $B$ another multiplication by 
$\;b*b'=(b\cdot F^1)(b'\cdot F^2)$ and denote this 
structure by $_FB$, then $_FB$ becomes a right $H_F$-module 
algebra with the same unit and $H$-action as for $B$ and we have 
an algebra isomorphism $H_F\# _FB\simeq H\# B$, $h\# b\mapsto 
hF^1\# b\cdot F^2$.  
\end{proposition}

We want to see that Proposition \ref{invquasi} is a particular case 
of Theorem \ref{main}. Since $H\# B$ is an associative algebra 
with unit $1_H\# 1_B$ and its multiplication satisfies 
$(h\#1_B)(h'\#b')=hh'\#b'$, for all $h, h'\in H$ and $b\in B$, 
it follows that $H\# B$ is a crossed product, namely 
$H\# B=H\ot _{R, \sigma }B$, where the maps $R$, $\sigma $ are 
defined, for all $b, b'\in B$ and $h\in H$, by
$R:B\ot H\rightarrow H\ot B$, $R(b\ot h)=h_1\ot b\cdot h_2$ and 
$\sigma :B\ot B\rightarrow H\ot B$, $\sigma (b, b')=x^1\ot 
(b\cdot x^2)(b'\cdot x^3)$.  
Similarly, we have 
$H_F\# _FB=H\ot _{R_F, \sigma _F}B$, where 
$R_F(b\ot h)=h_{(1)}\ot b\cdot h_{(2)}$ and 
$\sigma _F(b, b')=\tilde{x}^1\ot 
(b\cdot \tilde{x}^2)*(b'\cdot \tilde{x}^3)$. 

If we define $\theta , \gamma :B\rightarrow H\ot B$ by 
$\theta (b)=F^1\ot b\cdot F^2$ and 
$\gamma (b)=G^1\ot b\cdot G^2$, for all 
$b\in B$, then one can check that the hypotheses of 
Theorem \ref{main} are fulfilled for the crossed product 
$H\ot _{R, \sigma }B$ and the maps $\theta , \gamma $. It  
is also easy to see that  
the maps $R'$, $\sigma '$ given by Theorem \ref{main} 
coincide exactly to the maps $R_F$ and respectively 
$\sigma _F$, thus Theorem \ref{main} implies 
$H\ot _{R_F, \sigma _F}B\simeq H\ot _{R, \sigma }B$, that is 
$H_F\# _FB\simeq H\# B$ (and obviously the explicit 
isomorphisms given by Theorem \ref{main} and 
Proposition \ref{invquasi} coincide). 

We explain now briefly how Theorem \ref{main} 
generalizes the equivalence of crossed products 
by a coalgebra proved in \cite{brz}, Proposition 3.1. 
We recall first the framework in \cite{brz}. Let $P$ be an 
algebra, $C$ a coalgebra with a fixed group-like element $e$, 
let $(P, C, \Psi , e, \Psi ^C)$ be an 
{\em entwining data} (see \cite{brz} for the definition) 
and consider the algebra $A=P_e^{co (C)}$. Assume that 
there exist linear maps $f:C\ot C\rightarrow A$, 
$F:C\ot P\rightarrow P$, satisfying a certain set of conditions. 
If we define 
\begin{eqnarray*}
&&R_F:C\ot A\rightarrow A\ot C, \;\;\;R_F=(F\ot id_C)\circ (id_C\ot 
\Psi )\circ (\Delta \ot id_A), \\
&&\sigma _f:C\ot C\rightarrow A\ot C, \;\;\;\sigma _f=
(f\ot id_C)\circ (id_C\ot \Psi ^C)\circ (\Delta \ot id_C), 
\end{eqnarray*}
then $A\ot _{R_F, \sigma _f}C$ is a crossed product, denoted by 
$A\gsl _{F, f}C$ and called a {\em crossed product by a coalgebra} 
(see \cite{brz}, Proposition 2.2 for details). 

Assume that moreover $\nu :C\rightarrow A$ is a convolution 
invertible map (with convolution inverse denoted by $\nu ^{-1}$) 
such that $\nu (e)=1$ and 
\begin{eqnarray}
&&\Psi ^C_{23}\circ \Psi _{12}\circ (id_C\ot \nu \ot id_C)\circ 
(id_C\ot \Delta )
=(\nu \ot id_C\ot id_C)\circ (\Delta \ot id_C)\ot \Psi ^C. \label{tare}
\end{eqnarray}
Define maps $F_{\nu }:C\ot P\rightarrow P$, 
$f_{\nu }:C\ot C\rightarrow A$ by some explicit formulae 
given in \cite{brz}, Proposition 3.1. Then 
$A\gsl _{F_{\nu }, f_{\nu }}C$ is a crossed product 
by a coalgebra, isomorphic to $A\gsl _{F, f}C$  as algebras 
(this is the equivalence of crossed products by a coalgebra proved in 
\cite{brz}, Proposition 3.1).

We want to see that this result is a particular case of 
Theorem \ref{main}. In the hypotheses of \cite{brz}, Proposition 3.1, 
consider the linear maps $\theta , \gamma :C\rightarrow A\ot C$, 
$\theta (c)=\nu (c_1)\ot c_2$, $\gamma (c)=\nu ^{-1}(c_1)\ot c_2$. 
One can check that, for the maps $\theta $, $\gamma $ and the crossed 
product  $A\ot _{R_F, \sigma _f}C$, the hypotheses of 
Theorem \ref{main} are fulfilled (for (\ref{cros4}) one has to use 
the relation (\ref{tare})). Thus, we can apply Theorem \ref{main}, 
which gives the maps $R_F'$ and $\sigma _f'$ and the 
algebra isomorphism $A\ot _{R_F', \sigma _f'}C\simeq 
A\ot _{R_F, \sigma _f}C$. On the other hand, 
a straightforward computation shows that 
$R_{F_{\nu }}=R_F'$ and $\sigma _{f_{\nu }}=\sigma _f'$. Thus, 
from Theorem \ref{main} we obtain $A\gsl _{F_{\nu }, f_{\nu }}C=
A\ot _{R_{F_{\nu }}, \sigma _{f_{\nu }}}C=
A\ot _{R_F', \sigma _f'}C\simeq A\ot _{R_F, \sigma _f}C=
A\gsl _{F, f}C$, q.e.d.
%%%%%%%%%%%%%%%%%%%%%%%%

\end{document}